\documentclass{elsarticle}
\usepackage[english]{babel}
\usepackage{amsmath,amssymb,amsthm}
\usepackage{stmaryrd}
\usepackage{bbm}
\usepackage{color}

\newtheorem{theorem}{Theorem}[section]
\newtheorem{claim}[theorem]{Proposition}
\newtheorem{remark}[theorem]{Remark}

\newtheorem{corollary}[theorem]{Corollary}

\begin{document}

\newcounter{gl}
\title{Ergodic theorems for $L_1$--$L_\infty$ contractions in  Banach--Kantorovich $L_p$-lattices}

\author{V. I. Chilin }
\address{Department of Mathematics, National University of Uzbekistan,
Vuzgorodok, 100174, Tashkent, Uzbekistan}
\ead{chilin@ucd.uz}

\author{I.G. Ganiev}
\address{Department of Science in Engineering, Faculty of Engineering,%
  International Islamic University Malaysia, P.O. Box 10, 50728
Kuala Lumpur, MALAYSIA}
\ead{ganiev1@rambler.ru}

\date{}
\begin{abstract}
We present versions of  ergodic theorems for $L_1$--$L_\infty$ contractions in  Banach--Kantorovich $L_p$-lattices
associated with the Maharam measure taking values in the algebra
of measurable functions.
\end{abstract}


\maketitle

\section{Introduction}

Let $(\Omega,\Sigma,\mu)$ be a measure space with $\sigma$-finite
measure $\mu$ and let $L_p(\Omega,\Sigma,\mu)$ be the Banach space of
all measurable functions $f$ on $(\Omega,\Sigma,\mu)$ such that
 $\| f \|_p  = \left( \int\limits_{\Omega} | f |^p d\mu  \right)^{\frac{1}{p}}<\infty, \quad (1\leq p<\infty).$ The well-known M.A. Akcoglu's theorem \cite{Akc} asserts that
 for any positive contraction $T$ in the space $L_p(\Omega,\Sigma,\mu)$, $1 < p < \infty$, the averages $s_n(T)(f)=\frac{1}{n}
 \sum\limits_{i=0}^{n-1} T^i (f)$ converge almost everywhere for every $f\in L_p(\Omega,\Sigma,\mu)$, in addition $f^*=\sup\limits_{n\geq
1}s_n(T)(|f|)\in  L_p(\Omega,\Sigma,\mu)$ and $\| f^* \|_p\leq
\frac{p}{p-1}\|f\|_{p}.$ The same ergodic theorem holds for non
positive $L_1 - L_{\infty}$ contraction $T$ in
$L_p(\Omega,\Sigma,\mu)$ \cite[ch.VIII, \S 6]{DSh}. Further
development extends this ergodic  theorem to the space of
Banach-valued functions as follows (see e.g.
\cite[ch.4, \S 4.2]{Kr}): if  $T$ is an $L_1$--$L_\infty$ contraction
in the space $L_p(\Omega,X)$ of
 Bochner  maps from
$(\Omega,\Sigma,\mu)$ into a reflexive   Banach space $(X,
\|\cdot\|_X)$, $1<p<\infty,$ then there exists $\widetilde{f}\in
L_p(\Omega,X)$
 such that $\|s_n(T)(f)(\omega)-\widetilde{f}(\omega)\|_X\rightarrow
 0$  almost everywhere on $(\Omega,\Sigma,\mu)$ for any $f\in
L_p(\Omega,X)$.

The further development of the theory of vector measures
$m:\nabla\rightarrow E$
 on a complete Boolean algebra $\nabla$ with values in order
 complete vector lattices provides  new
 important examples of Banach--Kantorovich spaces \cite[ch. IV, V]{K2}. In particular, the $L_p$ - space $L_p(\nabla,m)$  associated with a
modular measure $m$, which admits as measurable bundle via
classical $L_p$--spaces \cite{Ga3}, is an example of such spaces.
It is natural to expect, that M.A. Akcoglu's  and N.Danford, J.T.
Schward's ergodic theorems are valid for
 the Banach~---Kantorovich spaces $L_p(\nabla,m)$. In \cite{CGa} the $(o)$--convergence
of averages $s_n(T)(f)$  in the vector lattice
$L_p(\nabla,m)$ has  been established  for positive contractions
$T:L_p(\nabla,m)\rightarrow L_p(\nabla,m)$, satisfying the condition
$T\mathbf{1}\leq\mathbf{1}$, where
$1<p<\infty$. In \cite{ZC} similar ergodic theorem has been
extended for positive contractions of the Orlicz--Kantorovich
lattices $L_M(\nabla,m)$ in case when $N$-function $M$ satisfies
the condition $\sup\limits_{s\geq 1}\frac{\int\limits _1^s
M(t^{-1}s)dt}{M(s)}<\infty.$

In the present paper we establish ergodic theorems for  $L_1$--$L_\infty$ contractions in  Banach--Kantorovich
 lattices $L_p(\nabla,m)$ and $L_M(\nabla,m)$.  Moreover, we present
"vector" versions of weighted and multiparameter weighted  ergodic theorems obtained in
\cite{JO}.

 To prove these  "vector"  versions of ergodic theorems can be use methods of the Boolean
 valued analysis.
By means of these methods, the Banach--Kantorovich lattices and
the corresponding bounded homomorphisms can be interpreted as the
Banach lattices and bounded linear mappings within a suitable
Boolean valued model of the set theory \cite[ch. XI]{K3}. This
approach to the theory of Banach--Kantorovich lattices makes it
possible to apply the transfer principle \cite[4.4]{K3} in order to
obtain various properties of Banach--Kantorovich lattices, which are similar to the corresponding properties of classical
Banach lattices. Naturally, when using such a method, an
additional study is needed to establish the "required"
interrelation between the objects of $\mathbf{2}$-valued and
Boolean valued models of the set theory.

 Another
important approach in the study  of Banach--Kantorovich spaces is
provided by the theory of continuous and measurable Banach bundles
\cite{Ga3}, \cite{G2}. The representation of a
Banach--Kantorovich lattice as a space of measurable sections of
a measurable Banach bundle makes it possible to obtain the required
properties of the lattice by means of the corresponding stalkwise
verification of the properties. Using this approach, a version
of the dominated ergodic theorem was obtained for positive
contractions in the $L_p(\nabla,m)$ \cite{CGa}. This approach we use in the present article. We use the theorem
that a Banach--Kantorovich lattice $L_p(\nabla,m)$ can be
represented as a measurable bundle of $L_p$ spaces associated with
scalar measures \cite{Ga3}. Then we apply the representation and
the corresponding ergodic theorems for $L_1$--$L_\infty$ contractions
in $L_p$--spaces in order to obtain  versions of ergodic theorems for
$L_1$--$L_\infty$ contractions in Banach--Kantorovich
lattices $L_p(\nabla,m)$.

 We
use the terminology and notation of the theory of Boolean
algebras, Riesz spaces, vector integration, and lattice-normed
spaces \cite{K2}, as well as the terminology of measurable bundles
of Boolean algebras and Banach lattices \cite{Ga3},\cite{G2}.

\section{Preliminaries}
Let $(\Omega,\Sigma,\mu)$ be a  measure space with the
direct sum property \cite[1.1.8]{K2}. Denote by
$\mathcal{L}(\Omega)$ (respectively, $\mathcal{L}_\infty(\Omega)$ ) the set of
all (respectively, essentially bounded) measurable real functions
defined a.e. on $\Omega$. Introduce an equivalence relation on
$\mathcal{L}(\Omega)$ by setting $f \sim g \Leftrightarrow  f = g$
a.e. The set $L_0(\Omega)$ of all cosets $f^\sim = \{g \in
\mathcal{L}(\Omega): f \sim g\}$ endowed with the natural
algebraic operations is an algebra over
the field $\mathbb{R}$ of real numbers with the unity $\bf{1}(\omega)$ = $1$. Moreover, with respect to the
partial order $f^\sim\leq g^\sim\Leftrightarrow  f \leq g$ a.e.,
the algebra $L_0(\Omega)$ is an order complete vector lattice with
weak unity $\bf{1}$, and the set $B(\Omega) :=
B(\Omega,\Sigma,\mu)$ of all idempotents in $L_0(\Omega)$ is a
complete Boolean algebra. Furthermore, $L_\infty(\Omega) =
\{f^\sim: f\in \mathcal{L}_\infty(\Omega)\}$ is an order ideal in
$L_0(\Omega)$ generated by $\bf{1}$, in addition, $L_\infty(\Omega)$ is a commutative Banach  algebra with respect to the norm $\|f^\sim\|_{\infty} = vraisup |f(\omega)|$. In what follows, we
write $f\in L_0(\Omega)$ instead of $f^\sim\in L_0(\Omega)$
assuming that the coset of $f$ is considered.

A mapping $l : L_\infty(\Omega)\rightarrow
\mathcal{L}_\infty(\Omega)$ is called a lifting of
$L_\infty(\Omega)$ if for all $\alpha,\beta\in
\mathbb{R}$  and $f, g\in  L_\infty(\Omega)$ the following conditions hold:

 (a) $l(f) \in f^\sim$ and
 $dom (l(f))=\Omega$, where $dom (g)$ is the domain of $g\in\mathcal{L}_\infty(\Omega)$;

  (b) if $f\leqslant g$ then $ l(f)\leqslant l(g)$ everywhere on $\Omega $;

  (c) $l(\alpha f+\beta g)=\alpha l(f)+\beta l(g)$,
 $l(fg)= l(f)l(g)$, $l(f\lor g)=l(f)\lor l(g)$, $l(f\land g)=l( f)\land l(g)$;

 (d) $l(0)=0$ and $l(\bold 1)= 1$ everywhere on
  $\Omega$.

Since $(\Omega,\Sigma,\mu)$ has a direct sum property then there always exists
lifting $l : L_\infty(\Omega)\rightarrow
\mathcal{L}_\infty(\Omega)$ \cite[1.4.8]{K2}.

 Let $\nabla$  be an arbitrary complete Boolean algebra, let
$X(\nabla)$ be the Stone space of $\nabla$,  let $L_0(\nabla) :=
C_\infty(X(\nabla))$ be the algebra of all continuous functions $x
: X(\nabla)\rightarrow  [-\infty,+\infty]$ taking the values
$\pm\infty$ only on nowhere dense subsets of $X(\nabla)$, and let
$L_\infty(\nabla):= C(X(\nabla))$ be the subalgebra of all continuous real functions
on $X(\nabla)$. It is clear that $L_\infty(\nabla)$ is a commutative Banach  algebra with respect to the norm $\|x\|_{\infty} = sup \{|x(t)|: t \in X(\nabla)\}$.

Let $\nabla_0$ be a regular Boolean subalgebra in $\nabla$, i.e.
$\sup A$ and $\inf A$ of a subset $A\subset\nabla_0$ calculated in
$\nabla_0$ coincide with those in $\nabla$. It is clear that
 $L_0(\nabla_0)$
is identified with a subalgebra of $L_0(\nabla)$; moreover,
$L_0(\nabla_0)$ is a regular sublattice of $L_0(\nabla)$, i.e.
the supremum and infimum of subsets of $L_0(\nabla_0)$ calculated in
$L_0(\nabla_0)$ coincide with those in $L_0(\nabla)$.

Let $\nabla$ and $B$ be  complete Boolean algebras. A mapping $m
: {\nabla}\to L_0(B)$ is called an  $L_0(B)$-valued measure if
$(i)$.  $m(e)\geq 0$ for all $e\in{\nabla}$ and
$m(e)=0\Leftrightarrow e=0$;   $(ii).$ $m(e\vee g)=m(e)+m(g)$ if  $\ e\wedge g=0, e,g\in{\nabla}$; $(iii)$.
$m(e_\alpha)\downarrow 0$ for every net $e_\alpha\downarrow 0$.

For any $p\geq 1$, by $L_p(\nabla,m)$ we denote the $(bo)$-- complete
lattice normed space of all  elements $x$ from $L_0(\nabla)$, for which there exists $L_0(B)$-valued norm
 $\|x\|_{p}=\bigg(\int| x|^pdm\bigg)^{1/p}$ (see for example
 \cite[6.1]{K2}).

 An $L_0(B)$-valued measure $m$ is said  to be  disjunctive
 decomposable ($d$-decomposable), if
for every $e\in \nabla$ and a decomposition $m(e)=a_1+a_2, \quad
a_1\wedge a_2=0, \quad a_i\in L_0(B)$ there exists $e_i\in
\nabla$ such that $e=e_1\vee e_2$ and $m(e_i)=a_i, i=1,2.$

Let $\nabla_0$ be a regular Boolean subalgebra of $\nabla$ and
$\varphi: B\rightarrow \nabla_0$ be an isomorphism from $B$ onto
$\nabla_0.$ A $L_0(B)$-valued measure $m$  is called
$\varphi$--modular, if $m(\varphi(q)e)=qm(e)$ for all $q\in B,
e\in\nabla.$ The following criteria of $d$-decomposability of
measures is well-known.

\begin{theorem}\cite{Cz1}
\label{th2_1}
 Let $\nabla$ and $B$ be complete
Boolean algebras. For an $L_0(B)$-valued measure $m$ the following
conditions are
equivalent:\\
 $(i)$. The measure $m$ is $d$-decomposable;\\
 $(ii)$. There exist a regular Boolean subalgebra $\nabla_0$ and an
 isomorphism $\varphi: B\rightarrow \nabla_0$  such that measure
 $m$ is $\varphi$--modular.
\end{theorem}

In case of $d$-decomposable measure $m$, the lattice normed space $(L_p(\nabla,m), \|\cdot\|_p)$
 has  additional properties.
 Let $\nabla_0$ and $B$ be the same as in Theorem \ref{th2_1} and let $\psi$ be an
 isomorphism from $L_0(B)$ onto $L_0(\nabla_0)$ with the restriction $\psi\mid_{B}$
 to the  Boolean algebra $B$ of idempotents in $L_0(B)$
 coinciding with $\varphi.$ By setting $\alpha\cdot
 x=\psi(\alpha)x$, we define multiplication $\alpha\cdot
 x$ for $\alpha\in L_0(B), x\in L_0(\nabla)$. It is clear
 that $L_0(\nabla)$ is a $L_0(B)$ module.
 In the next theorem some basic properties  of the space $L_p(\nabla,m)$
  are stated \cite{Cz}.

\begin{theorem}\label{th2_2}
 $(i)$ $(L_p(\nabla,m), \|\cdot\|_p)$ is a
$(bo)$--complete lattice normed ideal of
$L_0(\nabla)$, i.e. the condition $|x|\leq |y|$, $y\in L_p(\nabla,m)$,
$x\in L_0(\nabla)$ implies that $x\in L_p(\nabla,m)$ and
$\|x\|_p\leq\|y\|_p$;\\
$(ii)$ If $0\leq x_\alpha\in L_p(\nabla,m)$ and
$x_\alpha\downarrow 0$, then $\| x_\alpha\|_p\downarrow 0$;\\
$(iii)$ If the measure $m$ is $d$-decomposable,  $\nabla_0$ and
$\varphi$ are the same as in Theorem \ref{th2_1} and $\psi$ is an
 isomorphism from $L_0(B)$ onto $L_0(\nabla_0)$, such that
 $\psi\mid_B=\varphi$, then $L_0(\nabla_0)\cdot L_p(\nabla,m)\subset L_p(\nabla,m),$
 moreover  $\| \psi(\alpha)x\|_p=\alpha\| x\|_p$ for
 all $\alpha\in L_0(B), x\in L_p(\nabla,m)$; \, in addition,  $(L_p(\nabla,m), \|\cdot\|_p)$
 is a Banach--Kantorovich space;\\
 $(iv)$ $L_\infty(\nabla) \subset L_p(\nabla,m)\subset
 L_q(\nabla,m)$, $1\leq q\leq p,$ moreover $L_\infty(\nabla)$
 is  $(bo)$--dense in $(L_1(\nabla,m), \|\cdot\|_1).$
\end{theorem}

Now we mention  some necessary facts  from the theory of
measurable bundles of Boolean algebras and of Banach spaces.

 Let $X$ be a Banach bundle over $\Omega$, i.e., a mapping $\omega\mapsto X(\omega)$ from $\Omega$ into the class of Banach
spaces over field $\mathbb{R}$ of real numbers. Denote by
$S_\sim(\Omega,X )$ the set of all sections of $X$ that are
defined a.e. in $\Omega$. Let $ \mathcal{L}\subset S_\sim(\Omega,X
)$ be a measurable structure in $X$, i.e.,

(a) $\alpha_1 c_1+\alpha_2 c_2\in \mathcal{L}$ for all $\alpha_1,
\alpha_2\in \mathbb{R}$  and $c_1, c_2\in \mathcal{L}$;

 (b) the pointwise norm $|||c||| : \Omega\rightarrow R$ of each element $c\in\mathcal{L}$  is
measurable, where $|||c|||(\omega) = \|c(\omega)\|_{ X(\omega)}$;

(c) the set $\mathcal{L}$ is stalkwise dense in $X$, i.e.
$\{c(\omega) : c\in\mathcal{L}\}$ is dense in $X(\omega)$ for each
$\omega\in dom(c)$,  where  $dom(c)$ is the domain of $c$.

The pair $(X , \mathcal{L})$ is called a measurable Banach
bundle over $\Omega$. Denote by $\mathcal{M}(\Omega, X)$ the set
of all $\mathcal{L}$--measurable sections of the bundle $X$ and
let $L_0(\Omega, X)$  be the quotient of $\mathcal{M}(\Omega, X)$
by the equality almost everywhere. It is known that, under the
natural algebraic operations, the set $L_0(\Omega, X)$ endowed with the
norm $\|x^\sim\| := |||x|||^\sim\in L_0(\Omega), x^\sim\in
L_0(\Omega, X) $ is a Banach--Kantorovich space over
$L_0(\Omega)$ \cite{G2}.

A measurable Banach bundle $(X , \mathcal{L})$ over $\Omega$ is
called a measurable bundle of Banach lattices over $\Omega$, if
$X(\omega)$ is a Banach lattice for each $\omega\in\Omega$ and
$c_1\vee c_2 \in \mathcal{L}$ for all $c_1, c_2\in\mathcal{L}$. In
this case $L_0(\Omega, X)$, endowed with the natural partial order
$x^\sim \leq y^\sim \Leftrightarrow x(\omega)\leq y(\omega)$ for
almost all $\omega\in\Omega$, is a Banach--Kantorovich lattice
over $L_0(\Omega)$ \cite{Ga3}.

Let $(\Omega,\Sigma,\mu)$ be a  measure space with the
direct sum property. For given $f\in L_0(\Omega)$ denote by $s(f)$
the support of $f$, i.e. $s(f)=\mathbf{1}-\sup\{e\in
B(\Omega):fe=0\}$.
Let $\nabla$  be an arbitrary complete Boolean algebra and let $m$ be a
$d$-decomposable $L_0(\Omega)$-valued measure on $\nabla$.  By theorem \ref{th2_1} we can assume, that $B(\Omega)$ is a
regular subalgebra of $\nabla$ and $m$ is modular, i.e.
$m(qe)=qm(e)$ for all $q\in B(\Omega), e\in\nabla.$

 If $q=s(m(\mathbf{1}))$, then
 $m(\mathbf{1} - q)= m(\mathbf{1})-qm(\mathbf{1})=0$, i.e. $q=\mathbf{1}$.  Hence
 the element $m(\mathbf{1})$ is
invertible in $L_0(\Omega)$. It is clear that
  $m_1(e) = m(e)m(\mathbf{1})^{-1}$ is an $L_0(\Omega)$-valued measure on
 $\nabla$,
  for which we have
$m_1(\mathbf{1}) =  \mathbf{1}$. From now on we assume that $m(\mathbf{1}) =  \mathbf{1}$.

Let $l : L_\infty(\Omega)\rightarrow \mathcal{L}_\infty(\Omega)$
be a lifting. Define some scalar quasimeasure on $\nabla$ by the
equality $ m^0_\omega(e) = l(m(e))(\omega)$. For each
$\omega\in\Omega$, consider the ideal $I^0_\omega = \{e\in\nabla :
m^0_\omega(e) = 0\}$ in $\nabla$ and denote by $\nabla^0_\omega$  the
quotient Boolean algebra $\nabla/I^0_\omega$. It is clear that
$\nabla^0_\omega$ is a Boolean algebra and $m^0_\omega([e])
=m^0_\omega(e)$ is a strictly positive quasimeasure on
$\nabla^0_\omega$, where $[e]$ is the coset in $\nabla^0_\omega$
of $e\in\nabla$. Consider the metric $\rho_\omega([e], [g]) =
m^0_\omega(e\triangle g)$ in $\nabla^0_\omega$ and denote by
$\nabla_\omega$ the completion of the metric space
$(\nabla^0_\omega, \rho_\omega)$. It is known that $\nabla_\omega$
is a complete Boolean algebra with a strictly positive scalar
measure $m_\omega$ extending $m^0_\omega$  \cite[ch.III, \S 5]{Vl}.

Let $\pi_\omega : {\nabla}\rightarrow\nabla_\omega^0$ be the
quotient homomorphism, let $i_\omega :
\nabla_\omega^0\rightarrow\nabla_\omega$ be the natural embedding,
and let us set $\gamma_\omega = i_\omega\circ\pi_\omega$. Then
$\gamma_\omega$ is a homomorphism of the Boolean algebra $\nabla$
into the Boolean algebra $\nabla_\omega$ for each
$\omega\in\Omega$. Consider the bundle $X$  over $\Omega$ such
that $ X(\omega) = (\nabla_\omega,m_\omega)$ for all
$\omega\in\Omega$ and put $\mathcal{L} = \{e(\omega)
=\gamma_\omega(e) : e\in  \nabla\}$. It is known (see \cite{Ga3})
that $(X,\mathcal{L})$ is a measurable bundle of Boolean algebras
over $\Omega$ such that the Boolean algebras $\nabla$ and
$L_0(\Omega, X)$ are isometrically isomorphic; moreover,
$l(m(e))(\omega) = m_\omega(\gamma_\omega(e))$ for all $e\in
\nabla$  and $\omega\in\Omega$.

Consider the classical $L_p$--lattice
$L_p(\nabla_\omega,m_\omega)$ and the Banach~---Kantorovich
lattice $L_p(\nabla,m), p\geq1$. Every idempotent $e\in\nabla$
generate bundle $e(\omega)\in\nabla_\omega\subset
L_p(\nabla_\omega,m_\omega)$ by the equality
$e(\omega)=\gamma_\omega(e).$  Let $({Y}_p , \mathcal{E})$ be
Banach bundle over $\Omega$ such that ${Y}_p(\omega) =
L_p(\nabla_\omega,m_\omega).$   In \cite{Ga3} it is shown that
$$\mathcal{E}=\bigg\{\sum\limits_{i=1}^n \lambda_i
\gamma_\omega(e_i) : \lambda_i\in \mathbb{R},\ e_i\in\nabla,\
i=1,\ldots,n,\
 n\in\mathbb{N}\bigg\}$$ is a measurable structure in $Y_p$ and
 there exists  isometric isomorphism $\Phi$ from
 $L_p(\nabla,m)$
onto  the Banach--Kantorovich lattice $L_0(\Omega, Y_p)$ such that
$$\Phi\left(\sum\limits_{i=1}^n \lambda_i e_i\right)=\left(\sum\limits_{i=1}^n \lambda_i
\gamma_\omega(e_i)\right)^\sim$$ for all $\lambda_i\in
\mathbb{R},\ e_i\in\nabla, i=1,\ldots,n,\
 n\in\mathbb{N}.$ Thus, every element $x\in L_p(\nabla,m)$ is
 identified with the bundle $\Phi(x)\in L_0(\Omega, Y_p)$, where
 $\Phi(x)(\omega)\in L_p(\nabla_\omega,m_\omega)$ a.e..

Set
$$\mathcal{L}_\infty(\Omega, Y_p)=\{u\in \mathcal{M}(\Omega, Y_p):
|||u|||\in \mathcal{L}_\infty(\Omega)\}$$
  and
$$L_\infty(\Omega,
Y_p)=\{{u}^\sim: u \in \mathcal{L}_\infty(\Omega,Y_p)\}.$$ Let $l :
L_\infty(\Omega)\rightarrow \mathcal{L}_\infty(\Omega)$ be a
lifting. In \cite{Ga3} it is  proved that there exists a linear
mapping $\ell_\nabla: L_\infty(\nabla,m)\rightarrow
\mathcal{L}_\infty(\Omega,Y_p)$ such that for all $x,y\in
L_\infty(\nabla,m), h\in L_\infty(B(\Omega))$ the following properties  hold \\
\textit{\indent 1) $\ell_\nabla(x)\in \Phi(x),\ {\rm dom}~\ell_\nabla(x) = \Omega$;}\\
 \textit{\indent 2)
 $\|\ell_\nabla(x)(\omega)\|_{L_p(\nabla_\omega,m_\omega)}=
 l(\|x\|_p)(\omega)$ (we see that the equality $m(\mathbf{1})=\mathbf{1}$ implies that $\|x\|_p\in L_\infty(\Omega)$ for all $x\in
L_\infty(\nabla,m)$);}\\
 \textit{\indent 3) $\ell_\nabla(x)(\omega)\geq 0$ if $x\geq 0$;}\\
 \textit{\indent 4) $\ell_\nabla(hx) = l(h)\ell_\nabla(x)$;}\\
\textit{ \indent 5) $\{\ell_\nabla(x)(\omega) : x\in L_\infty(\nabla,m)\}$ is dense in  $L_p(\nabla_\omega,m_\omega), \omega\in\Omega$;}\\
 \textit{\indent 6)
 $\ell_\nabla(x\vee y)=\ell_\nabla(x)\vee\ell_\nabla(y)$.}\par

A mapping
$\ell_\nabla: L_\infty(\nabla,m)\rightarrow
\mathcal{L}_\infty(\Omega,Y_p)$ is called  \textit{ a vector
valued lifting}   associated with the lifting $l :
L_\infty(\Omega)\rightarrow \mathcal{L}_\infty(\Omega)$.

Let $(E, \|\cdot\|_E)$ be a Banach--Kantorovich lattice over
$L_0(\Omega)$.
 A linear mapping $T : E\rightarrow E$  is called
\begin{description}
\item[---]
positive, if $Tx \geq 0$ for all $x\geq 0$;
\item[---]
$L_0(\Omega)$-bounded, if there exists $0\leq c\in L_0(\Omega)$
such that $\|Tx\|_E\leq c\|x\|_E$ for all $x\in E$ (in this case
we set $\|T\| := \|T\|_{E\rightarrow E} = \sup\{\|Tx\|_E :
\|x\|_E\leq \mathbf{1}\}$).
\end{description}

In \cite[2.1.8]{K2} it is shown that in a Banach--Kantorovich
lattice $(E, \|\cdot\|_E)$ over $L_0(\Omega)$ can be  defined
multiplication operation $\alpha x, \alpha\in L_0(\Omega), x\in E$,
such that $E$ become a  Banach $L_0(\Omega)$--module, with an
additional property $\|\alpha x\|_E=|\alpha| \|x\|_E$, moreover
any $L_0(\Omega)$--bounded linear mapping $T : E\rightarrow E$
become an $L_0(\Omega)$--linear, i.e. $T(\alpha x)=\alpha T(x)$ for
all $\alpha\in L_0(\Omega), x\in E$ \cite[5.1.9]{K2}.

 In that case, when $E=L_p(\nabla,m), \|\cdot\|_E=\|\cdot\|_p,$ \  $p\geq1,$
 and $m:\nabla\rightarrow L_0(\Omega)$ is a $d$--decomposable
 measure, the  equality $\| \varphi(\alpha) x\|_p=\alpha\| x\|_p$ hold for
 all $\alpha\in L_0(\Omega), x \in L_p(\nabla,m)$ (see Theorem
 \ref{th2_2}(iii)), and therefore the $L_0(\Omega)$--linearity of an
 $L_0(\Omega)$--bounded linear mapping $T : L_p(\nabla,m)\rightarrow
L_p(\nabla,m)$ means that
$T(\varphi(\alpha) x)=\alpha T(x)$ for
 all $\alpha\in L_0(\Omega), x \in L_p(\nabla,m)$.
 Since we have identified the Boolean algebra $B(\Omega)$ with regular
 subalgebra of $\nabla$ and we assume that $m(qe)=qm(e)$   for
 all $q\in B(\Omega),  e\in\nabla$, it follows that $T : L_p(\nabla,m)\rightarrow L_p(\nabla,m)$ is
 $L_0(\Omega)$--linear if and only if
$T(\alpha x)=\alpha T(x)$ for
 all $\alpha\in L_0(B(\Omega)), x \in L_p(\nabla,m)$.

Further we will need the following stalkwise representation
of  $L_0(\Omega)$--bounded linear operators acting
in Banach~---Kantorovich lattices $L_p(\nabla,m)$.

 Let $T: L_1(\nabla,m)\rightarrow
 L_1(\nabla,m)$ be an  $L_0(\Omega)$-- bounded linear operator, $\|T\|_{L_1(\nabla,m)\rightarrow
L_1(\nabla,m)}\leq\mathbf{1}$, $T(L_\infty(\nabla,m))\subset
L_\infty(\nabla,m)$ and let $\ell_\nabla:
L_\infty(\nabla,m)\rightarrow \mathcal{L}_\infty(\Omega,Y_1)$ be the
vector valued lifting   associated with a lifting $l :
L_\infty(\Omega)\rightarrow \mathcal{L}_\infty(\Omega)$. We define a mapping
 $\varphi(\omega)$ from  $\{\ell_\nabla(x)(\omega) : x\in
 L_\infty(\nabla,m)\}$ into
 $L_1(\nabla_\omega,m_\omega)$ by the equality  $\varphi(\omega)
 (\ell_\nabla(x)(\omega))=\ell_\nabla(Tx)(\omega), \omega\in\Omega$. By $\|Tx\|_1\leq
 \|x\|_1$ we have that
 $$\|\ell_\nabla(Tx)(\omega)\|_{L_1(\nabla_\omega,m_\omega)}=
 l(\|Tx\|_1)(\omega)\leq l(\|x\|_1)(\omega)
 = \|\ell_\nabla(x)(\omega)\|_{L_1(\nabla_\omega,m_\omega)},$$
and therefore the operator $\varphi(\omega)$ is well defined   and
  bounded with respect to the norm  $\|\cdot\|_{L_1(\nabla_\omega,m_\omega)}$.
 Since $\{\ell_\nabla(x)(\omega) :
 x\in L_\infty(\nabla,m)\}$ is dense in
 $L_1(\nabla_\omega,m_\omega)$, then the linear operator
 $\varphi(\omega)$ can  extended  to the  contraction  $T_\omega : L_1(\nabla_\omega,m_\omega) \rightarrow L_1(\nabla_\omega,m_\omega)$.

 We shall show that $\Phi(Tx)(\omega)=T_\omega(\Phi(x)(\omega))$
 for a.e. $\omega\in\Omega$, where $x\in L_1(\nabla,m)$. Choose $\{x_n\}\in L_\infty(\nabla,m)$ such
 that the sequence$\|x_n-x\|_1$\, $(o)$-converges to zero.
Then
$\|\Phi(x_n)(\omega)-\Phi(x)(\omega)\|_{L_1(\nabla_\omega,m_\omega)}\rightarrow
0$ for a.e. $\omega\in\Omega$. Since
$\|Tx_n-Tx\|_1\stackrel{(o)}\rightarrow  0$, it follows that
$\|\ell_\nabla(Tx_n)(\omega)-\Phi(Tx)(\omega)\|_{L_1(\nabla_\omega,m_\omega)}\rightarrow
0$ for a.e. $\omega\in\Omega$. Moreover, the continuity of the
operator $T_\omega$ implies that
$\|\ell_\nabla(Tx_n)(\omega)-T_\omega(\Phi(x))(\omega)\|_{L_1(\nabla_\omega,m_\omega)}\rightarrow
0$ for a.e. $\omega\in\Omega$. Hence
$\Phi(Tx)(\omega)=T_\omega(\Phi(x)(\omega))$
 for a.e. $\omega\in\Omega$. It is clear that for the positive
 operator $T$, by propertie 3) of the vector valued lifting
$\ell_\nabla$, the operator $T_\omega$ is also positive.

Thus we obtain the following theorem.
\begin{theorem}\label{th2_3}
Let $T: L_1(\nabla,m)\rightarrow
 L_1(\nabla,m)$ be an $L_0(\Omega)$-- bounded linear operator, $\|T\|_{L_1(\nabla,m)\rightarrow
L_1(\nabla,m)}\leq\mathbf{1}$, $T(L_\infty(\nabla,m))\subset
L_\infty(\nabla,m)$. Then for every $\omega\in\Omega$ there exists a
contraction $T_\omega:L_1(\nabla_\omega,m_\omega)\rightarrow
L_1(\nabla_\omega,m_\omega)$ such that
$\Phi(Tx)(\omega)=T_\omega(\Phi(x)(\omega))$
 for a.e. $\omega\in\Omega$ and for every $x\in L_1(\nabla,m)$.
 In addition, if the
 operator $T$ is positive, then the operator $T_\omega$ is also positive for every $\omega\in\Omega$.
\end{theorem}

 A linear operator  $T : L_1(\nabla,m)\rightarrow L_1(\nabla,m)$ is
said to be  regular if it can be represented as a
difference of two positive operators. The set of all regular
operators on $L_1(\nabla,m)$ is denoted by $H_r(L_1(\nabla,m))$.
It is known that $H_r(L_1(\nabla,m))$ forms an order complete vector
lattice, in addition for every  $T\in H_r(L_1(\nabla,m))$ the
module $|T|$ is a positive linear operator and

$$|T|(x)=\sup\{|T{y}|: {y}\in L_1(\nabla,m),
|{y}|\leq {x}\ \}$$ where $0\leq x\in L_1(\nabla,m)$ \cite[3.1.2]{K2}. In addition $$|T{x}|\leq |T||{x}|$$ for all ${x}\in
L_1(\nabla,m).$

The set of all
$L_0(\Omega)$-- bounded linear operators acting in the Banach-- Kantorovich
lattice $L_1(\nabla,m)$ we denote by $B(L_1(\nabla,m))$. With
respect to the $L_0(\Omega)$--valued norm
$\|\cdot\|_{L_1(\nabla,m)\rightarrow L_1(\nabla,m)}$
this spase is a Banach--Kantorovich space \cite[4.2.6]{K2}.

We need the following  property of regularity for operators $T\in
B(L_1(\nabla,m))$.

\begin{claim}\label{pr2_4} $B(L_1(\nabla,m))\subset
H_r(L_1(\nabla,m)).$
\end{claim}

\begin{proof}  Let $T\in B(L_1(\nabla,m))$, $0\leq x\in
L_1(\nabla,m).$  The set of all elements from $L_1(\nabla,m)$
of the form $y=|T(x_1)|+\cdots+|T(x_n)|$ is denoted by $E(x)$,
where $x=x_1+\cdots+x_n, x_i\geq0, i=1,2,...,n.$ It is clear that
for $y\in E(x)$ the following inequalities hold
$$\|y\|_1\leq \sum\limits_{i=1}^{n}\|Tx_i\|_1\leq \|T\|\sum\limits_{i=1}^{n}\|x_i\|_1=\|T\|\|x\|_1.$$
 Repeating the proof of \cite[Theorem VIII.7.2]{Vul}
 we obtain that for any $y_1,y_2,\cdots y_k$ from $E(x)$
there exists $y\in E(x)$ such that $\sup\limits_{1\leq i\leq
k}y_i\leq y.$
Since $\|y\|_1\leq \|T\|\|x\|_1$ we have that $\|\sup\limits_{1\leq i\leq
k}y_i\|_1\leq \|T\|\|x\|_1.$

We denote by $A$ the direction all of finite  subsets of $E(x)$,
ordered by inclusion and for every  $\alpha\in A$ we set
$y_\alpha=\sup\{y: y\in\alpha\}$. It is clear that
$\{y_\alpha\}_{\alpha\in A}$ is an increasing  net of positive
elements from $L_1(\nabla,m),$ in addition $\|y_\alpha\|_1\leq
\|T\|\|x\|_1$ for all $\alpha\in A.$ By theorem of
monotone convergence \cite{zak} there exists $z\in
L_1(\nabla,m)$ such  that $y_\alpha\uparrow z$. Hence $E(x)$ is an order
bounded set in $L_1(\nabla,m)$.

Repeating again the proof of \cite[Theorem VIII.7.2]{Vul}
we have that $T(F)$ is an order bounded set in $L_1(\nabla,m)$ for any
order bounded set $F\subset L_1(\nabla,m)$. Therefore, by \cite[Theorem VIII 2.2]{Vul}, $T\in H_r(L_1(\nabla,m)).$
\end{proof}

 Proposition \ref{pr2_4} implies the following result.
\begin{theorem}\label{th2_5} Let $T: L_1(\nabla,m)\rightarrow
 L_1(\nabla,m)$ be an  $L_0(\Omega)$-bounded linear operator  in
 $L_1(\nabla,m)$. Then there exists a unique    $L_0(\Omega)$-bounded linear positive operator $|T|$ in  $L_1(\nabla,m)$ such that
\begin{enumerate}
\item[(i)] $\|T\|=\| |T| \|;$

\item[(ii)] $|T{x}|\leq
|T||{x}|$ for all ${x}\in L_1(\nabla,m)$;

\item[(iii)] $|T|{x}=\sup\{|T{y}|: {y}\in L_1(\nabla,m),
|{y}|\leq {x}\ \}$ for all ${x}\in L_1(\nabla,m), {x}\geq0$;

\item[(iv)] If \ \ $T(L_\infty(\nabla,m))\subset L_\infty(\nabla,m)$
and $\|T\|_{L_\infty(\nabla,m)\rightarrow
L_\infty(\nabla,m)}<\infty$, then $|T|(L_\infty(\nabla,m))\subset
L_\infty(\nabla,m)$ and $\|T\|_{L_\infty(\nabla,m)\rightarrow
L_\infty(\nabla,m)}=\| |T| \|_{L_\infty(\nabla,m)\rightarrow
L_\infty(\nabla,m)}$.
 \end{enumerate}
\end{theorem}

\begin{proof}
The existence of $|T|$ follows from Proposition \ref{pr2_4}.
Properties (ii), (iii) follow from the definition of $|T|$. The
uniqueness follows from property (iii). The equality  (i) can be
established  as in \cite[Theorem VIII. 6.3]{Vul}.

(iv). Let $y\in L_\infty(\nabla,m)$ and $|y|\leq \mathbf{1}$. Then
$\|y\|_\infty\leq 1$ and, since $T(L_\infty(\nabla,m))\subset
L_\infty(\nabla,m)$ we have that
$$|Ty|\leq\|Ty\|_\infty\mathbf{1}\leq
\|T\|_{L_\infty(\nabla,m)\rightarrow
L_\infty(\nabla,m)}\|y\|_\infty \mathbf{1}\leq
\|T\|_{L_\infty(\nabla,m)\rightarrow L_\infty(\nabla,m)}
\mathbf{1}.$$
 Hence, $$|T|(\mathbf{1})\leq
\|T\|_{L_\infty(\nabla,m)\rightarrow L_\infty(\nabla,m)}
\mathbf{1}.$$

Since the operator $|T|$ is positive, it follows that
$$-\|x\|_\infty|T|(\mathbf{1})\leq |T|(x)\leq \|x\|_\infty |T|(\mathbf{1})$$
for all $x\in L_\infty(\nabla,m)$, i.e.
 $$||T|(x)|\leq \|x\|_\infty |T|(\mathbf{1})\leq
\|x\|_\infty\|T\|_{L_\infty(\nabla,m)\rightarrow
L_\infty(\nabla,m)} \mathbf{1}.$$
Hence
$|T|(L_\infty(\nabla,m))\subset L_\infty(\nabla,m).$
Further, the inequality  $|T{x}|\leq |T||{x}|$ implies that
$$\|T\|_{L_\infty(\nabla,m)\rightarrow L_\infty(\nabla,m)}\leq\|
|T| \|_{L_\infty(\nabla,m)\rightarrow L_\infty(\nabla,m)}.$$
 On the other hand   $$|T|(|x|)=\sup\{|T{y}|: {y}\in
L_\infty(\nabla,m), |{y}|\leq |{x}| \}\leq
 \|T\|_{L_\infty(\nabla,m)\rightarrow
L_\infty(\nabla,m)}\|x\|_\infty\mathbf{1},$$ i.e.
$\||T|\|_{L_\infty(\nabla,m)\rightarrow L_\infty(\nabla,m)}\leq
\|T\|_{L_\infty(\nabla,m)\rightarrow L_\infty(\nabla,m)}.$
\end{proof}

\section{Ergodic theorems for $L_1$--
 $L_\infty$ contractions in   $L_p(\nabla,m)$}
 Let $\nabla$ be an arbitrary complete Boolean algebra,
 let $B(\Omega)$ be a complete Boolean algebra of all idempotents in
 $L_0(\Omega,\Sigma,\mu)$, where $(\Omega,\Sigma,\mu)$ is a  measure space with the direct sum property, let $m : {\nabla}\to
L_0(\Omega,\Sigma,\mu)$ be a $d$--decomposable measure. We shall
identify $B(\Omega)$ with regular Boolean subalgebra in $\nabla$
and assume  that $m(qe)=qm(e)$ for all $q\in B(\Omega),
e\in\nabla.$

A linear operator $T : L_1(\nabla,m)\rightarrow
 L_1(\nabla,m)$ is called an $L_1$-- $L_\infty$ contraction if $T\in
 B(L_1(\nabla,m))$, $T(L_\infty(\nabla,m))\subset L_\infty(\nabla,m)$
and $\|T\|_{L_1(\nabla,m)\rightarrow
L_1(\nabla,m)}\leq\mathbf{1}$,
$\|T\|_{L_\infty(\nabla,m)\rightarrow L_\infty(\nabla,m)}\leq1.$
The set of all $L_1$-- $L_\infty$ contractions we denote by
$C_{1,\infty}(\nabla,m).$ Theorem \ref{th2_5} implies that $|T|\in
C_{1,\infty}(\nabla,m)$ for any  $T\in C_{1,\infty}(\nabla,m)$.

\begin{theorem}\label{th3_1} If  $T\in C_{1,\infty}(\nabla,m)$, $p>1$, then $T(L_p(\nabla,m)) \subset L_p(\nabla,m)$
 and $\|T\|_{L_p(\nabla,m)\rightarrow
L_p(\nabla,m)}\leq\mathbf{1}.$
\end{theorem}
\begin{proof}
By Theorem  \ref{th2_3} for every
$\omega\in\Omega$
 there exists a positive linear contraction $S_\omega :
 L_1(\nabla_\omega,m_\omega)\rightarrow L_1(\nabla_\omega,m_\omega)$ such that $S_\omega (\Phi(x)(\omega)) = \Phi(|T|x)(\omega)$ for a.e.
 $\omega\in\Omega$ and for every
 $x\in L_1(\nabla,m)$. Since  $|T|\in C_{1,\infty}(\nabla,m)$, it follows that
$|T|(\mathbf{1})\leq \mathbf{1}$, and therefore
$S_\omega\mathbf{1}_{\nabla_\omega}=\Phi(|T|\mathbf{1})(\omega)\leq
 \Phi(\mathbf{1})(\omega)=\mathbf{1}_{\omega}$ for every
$\omega\in\Omega$, \  where $\mathbf{1}_{\omega}$ is the unit element
of the Boolean algebra $\nabla_\omega$. Hence  $S_\omega$ is a
 positive linear contraction in $L_\infty(\nabla_\omega,m_\omega)$ for a.e.
 $\omega\in\Omega$.

 Since $ L_p(\nabla_\omega,m_\omega)$ is an interpolation space  between  $
 L_1(\nabla_\omega,m_\omega)$ and
 $L_\infty(\nabla_\omega,m_\omega)$ \cite[ch. II, \S 4]{kre}, we have that $S_\omega(L_p(\nabla_\omega,m_\omega))\subset
  L_p(\nabla_\omega,m_\omega)$ and \\ $\|S_\omega\|_{L_p(\nabla_\omega,m_\omega)\rightarrow L_p(\nabla_\omega,m_\omega)}\leq
 1$. Hence $|T|(L_p(\nabla,m))\subset L_p(\nabla,m)$. Since
 $\Phi(|x|^p)(\omega)=(\Phi(|x|)(\omega))^p$  for a.e.
 $\omega\in\Omega$  ($x\in L_p(\nabla,m)$) \cite{Ga3}, it follows that
  $$\||T|(|x|)\|^p_{p}(\omega) = \|S_\omega(\Phi(|x|)(\omega))\|^p_{L_p(\nabla_\omega,m_\omega)}   \leq \|\Phi(|x|)(\omega)\|^p_{L_p(\nabla_\omega,m_\omega)} = \|x|^p_{p}(\omega)$$
for a.e.   $\omega\in\Omega$  \cite{Ga3}.  Hence
$\||T|\|_{L_p(\nabla,m)\rightarrow L_p(\nabla,m)}\leq\mathbf{1}.$
The inequality $|T{x}|\leq
|T||{x}|$ implies that
$T(L_p(\nabla,m)) \subset L_p(\nabla,m)$, in addition
$\|T\|_{L_p(\nabla,m)\rightarrow
L_p(\nabla,m)}\leq\||T|\|_{L_p(\nabla,m)\rightarrow
L_p(\nabla,m)}\leq\mathbf{1}.$
\end{proof}

The following theorem is a version of Theorem  \ref{th2_3} for an operator $T\in
C_{1,\infty}(\nabla,m).$
\begin{theorem}\label{th3_2} If  $T\in C_{1,\infty}(\nabla,m)$, then for
every $\omega\in\Omega$ there exists  $T_\omega \in C_{1,\infty}(\nabla_\omega,m_\omega)$
 such that $T_\omega(\Phi(x)(\omega)) = \Phi(Tx)(\omega)$ for a.e.
 $\omega\in\Omega$ and  for every
$x\in L_1(\nabla,m)$.
\end{theorem}
\begin{proof}Theorem
 \ref{th2_3} provides the existence of a linear operator   $T_\omega$ in $L_1(\nabla_\omega,m_\omega)$
 satisfying $T_\omega(\Phi(x)(\omega)) = \Phi(Tx)(\omega)$, in addition
 $T_\omega(\ell_\nabla(x)(\omega))=\ell_\nabla(Tx)(\omega)$ for
 all $\omega\in\Omega$, $x\in L_\infty(\nabla,m)$. Let $S_\omega$
 be
 those positive linear contractions in
 $L_1(\nabla_\omega,m_\omega)$, that in proof of Theorem  \ref{th3_1}.
 Since $|T{x}|\leq |T||{x}|$ and $S_\omega(\ell_\nabla(x)(\omega))=\ell_\nabla(|T|x)(\omega)$
 for
 all $\omega\in\Omega$, $x\in L_\infty(\nabla,m)$ we have that $$|T_\omega(\ell_\nabla(x)(\omega))|=|\ell_\nabla(T(x))(\omega)|\leq\ell_\nabla(|T||x|)(\omega)=
 S_\omega(\ell_\nabla(|x|)(\omega))=S_\omega(|\ell_\nabla(x)(\omega)|).$$
 Using density of the linear space $\{\ell_\nabla(x)(\omega) :
 x\in L_\infty(\nabla,m)\}$  in
 $L_1(\nabla_\omega,m_\omega)$ (see property 5) of $\ell_\nabla$)  we
  obtain that $|T_\omega g|\leq S_\omega|g|$ for all $g\in L_1(\nabla_\omega,m_\omega).$
Since every bounded linear operator in $L_1(\nabla_\omega,m_\omega)$
is regular, the   module $|T_\omega|$ is defined, which is a
positive contraction in $L_1(\nabla_\omega,m_\omega)$, in addition $$|T_\omega|h=\sup\{|T_\omega{g}|: {g}\in
L_1(\nabla_\omega,m_\omega), |{g}|\leq {h}\ \}\leq S_\omega h$$
for all $0\leq{h}\in L_1(\nabla_\omega,m_\omega).$ In particular
$|T_\omega|(\mathbf{1}_\omega)\leq S_\omega(\mathbf{1}_\omega)\leq
\mathbf{1}_\omega,$ that implies
$|T_\omega|(L_\infty(\nabla_\omega,m_\omega))\subset
L_\infty(\nabla_\omega,m_\omega)$ and
$\||T_\omega|\|_{L_\infty(\nabla_\omega,m_\omega)\rightarrow
L_\infty(\nabla_\omega,m_\omega)}\leq 1$. Since $|T_\omega
g|\leq|T_\omega||g|$ for all ${g}\in L_1(\nabla_\omega,m_\omega)$,
it follows that $T_\omega(L_\infty(\nabla_\omega,m_\omega))\subset
L_\infty(\nabla_\omega,m_\omega)$ and
$\|T_\omega\|_{L_\infty(\nabla_\omega,m_\omega)\rightarrow
L_\infty(\nabla_\omega,m_\omega)}\leq 1.$
\end{proof}

We provide some auxiliary facts related to $(o)$-convergence of a
sequence $\{x_n\}\subset L_1(\nabla,m)$ and $(o)$-convergence of
the sequence $\{\Phi(x_n)\}(\omega)\}$, $\omega\in\Omega$.

Consider on the vector lattice $L_0(\nabla,m)$ the metric
$\rho(x,y)=\allowbreak
 \int |x-y|(\mathbf{1}+|x-y|)^{-1}d{m}$ with values in $L_0(\Omega)$.

 Let  $Z:\omega\rightarrow
Z(\omega)=L_0(\nabla_\omega,m_\omega)$ be a bundle over $\Omega$
of metric spaces
 $(L_0(\nabla_\omega,m_\omega),\rho_\omega),$ where $\rho_\omega(u(\omega),v(\omega))=\allowbreak
 \int
 |u(\omega)-v(\omega)|(\mathbf{1}_{\omega}+|u(\omega)-v(\omega)|)^{-1}d{m_\omega}$,
 $\omega\in\Omega$. In \cite{Ga3} it is established that there exists an
 isometric isomorphism $\Psi$ from $(L_0(\nabla,m),\rho)$ onto
 complete $L_0(\Omega)$--metrisable vector lattice $L_0(\Omega,
(Z,\mathcal{E}))$ such that
$$\Psi\left(\sum\limits_{i=1}^{n}\lambda_ie_i\right)= \left(\sum\limits_{i=1}^{n}\lambda_i\gamma_\omega(e_i)\right)^\sim= \Phi\left(\sum\limits_{i=1}^{n}\lambda_ie_i\right),$$
in addition, $L_0(\Omega,(Y_1, \mathcal{E}))$ can be identified with vector sublattice
in $L_0(\Omega, (Z,\mathcal{E}))$ and $\Psi(x)=\Phi(x)$ for all
$x\in L_1(\nabla,m)$, where $Y_1(\omega) = L_1(\nabla_\omega,m_\omega)$.

\begin{theorem}\label{th3_3}  \cite{CGa}. If  $\{x_{{n}}\}\subset
L_0(\nabla,m)$, then $\sup\limits_{n\geq 1}x_{n}$ exists in
$L_0(\nabla,m)$ if and only
 if $\sup\limits_{n\geq
1}\Psi(x_n){(\omega)}$ exists in $L_0(\nabla_\omega,m_\omega)$ for
 a.e. $\omega\in\Omega$. In this case,
 $\Psi(\sup\limits_{n\geq 1}x_{n})(\omega)=\sup\limits_{n\geq 1}\Psi(x_n){(\omega)}$ for
 a.e. $\omega\in\Omega$.
\end{theorem}

This theorem  implies the following corollary.
\begin{corollary}\label{c3_4} \cite{CGa}. If ${x}_{{n}}, {x}\in
L_0(\nabla,m)$ and ${x}_{{n}} \stackrel{(o)}\rightarrow {x}$,
 then
 $\Psi({x}_{{n}})(\omega) \stackrel{(o)}\rightarrow \Psi({x})(\omega)$ in $L_0(\nabla_\omega,m_\omega)$
 for  a.e. $\omega\in\Omega$. Conversely, if ${x}_{{n}}\in L_0(\nabla,m)$ and $\Psi(x_{{n}})(\omega)
  \stackrel{(o)}\rightarrow v(\omega)$ for some $v(\omega)\in
 L_0(\nabla_\omega,m_\omega)$ for  a.e. $\omega\in\Omega$,
 then there exists  $x\in L_0(\nabla,m)$ such that $\Psi({x})(\omega)=v(\omega)$ for
 a.e. $\omega\in\Omega$  and ${x}_n \stackrel{(o)}\rightarrow
 {x}$ in $L_0(\nabla,m)$.
\end{corollary}

The following theorem is a vector version of
 well known N.Danford and J.T. Schward's ergodic theorems
for  $L_1$--$L_\infty$ contraction in Banach--Kantorovich
lattice
 $L_p(\nabla,m)$ associated with $L_0(\Omega)$--valued measure.

\begin{theorem}\label{th3_5} If  $T\in C_{1,\infty}(\nabla,m)$, $s_n(T)({x})=\frac{1}{n}
 \sum\limits_{i=0}^{n-1} T^i ({x})$, $x\in L_p(\nabla,m)$,  $1 \leq p < \infty$, \  then
\begin{enumerate}
\item[(i)] (statistical ergodic theorem)\, the sequence  $\{s_n(T)({x})\}$ is  $(bo)$ - convergent in   $L_p(\nabla, m)$   for any ${x} \in L_p(\nabla, m)$ ;

\item[(ii)] for every  ${x}\in L_p(\nabla, m), p>1,\ q>1,\
 \frac{1}{p}+\frac{1}{q}=1,$ the sequence  $s_n(T)({x})$ is order bounded in
 $L_p(\nabla, m)$ and
 $\|\sup\limits_{n\geq 1} |s_n(T)(x)|\|_p \leq q
 \|{x}\|_p$, in this case there exists  $\tilde{x}\in
 L_p(\nabla, m)$ such that the sequence $s_n(T)({x})$ is $(o)$-convergent to $\tilde{x}$
 in
 $L_p(\nabla, m)$;

\item[(iii)] for every ${x}\in L_1(\nabla, m)$ there exists $\tilde{x}\in
 L_1(\nabla, m)$ such that the sequence
 $s_n(T)({x})$ is $(o)$-convergent to $\tilde{x}$ in
 $L_0(\nabla, m)$.
 \end{enumerate}
\end{theorem}
\begin{proof}
From the proof of Theorem \ref{th3_1} it follows that $|T|$ is a positive contraction in $L_p(\nabla,m)$, in addition $|T|(\mathbf{1})\leq \mathbf{1}.$
 From \cite{CGa} follows correctness items (i)-(iii) of the theorem for the operator  $|T|\in C_{1,\infty}(\nabla,m)$.
Since $|T^ix|\leq|T|^i(|x|), i=1,2,...$ (see Theorem \ref{th2_5}), it follows that
for $x\in L_p(\nabla,m)$ the inequalities
$$|s_n(T)(x)|\leq \frac{1}{n}
 \sum\limits_{i=0}^{n-1} |T^i (x)|\leq \frac{1}{n}
 \sum\limits_{i=0}^{n-1} |T|^i ||x|=s_n(|T|)(|{x}|)$$
holds.

Since  Theorem \ref{th3_5} (ii) is valid for $|T|$, the sequence
 $\{s_n(|T|)(|x|)\}$ is order bounded in $L_p(\nabla,m)$ and
 $$\|\sup\limits_{n\geq 1}
 |s_n(T)(x)|\|_p\leq\|\sup\limits_{n\geq 1}
 s_n(|T|)(|{x}|)\|_p\leq q \|{x}\|_p.$$

 According to  Theorem \ref{th3_1} and \ref{th3_2} we have that
 $s_n(T_\omega)(\Phi(x)(\omega))=\Phi(s_n(T)(x))(\omega)$ for
 a.e. $\omega\in\Omega$, where ${x}\in L_p(\nabla, m)$. Since
 $T_\omega\in C_{1,\infty}(\nabla_\omega,m_\omega)$ (Theorem \ref{th3_2}),
 Corollaries 4 and 5 \cite[ch.VIII, \S 5]{DSh} imply that
 there exists  $v(\omega)\in L_p(\nabla_\omega,m_\omega)$ such that
$$\|s_n(T_\omega)(\Phi(x)(\omega))-v(\omega)\|_{L_p(\nabla_\omega,m_\omega)}\rightarrow
0$$ as $n\rightarrow\infty.$ Since $\Phi(s_n(T))\in L_0(\Omega,Y_p)$, it follows that $v^\sim\in L_0(\Omega,Y_p)$ and there exists
$\tilde{x}\in L_p(\nabla, m)$ such that $\Phi(\tilde{x})=v^\sim$.
Therefore, $\|s_n(T)(x)- \tilde{x}\|_p\stackrel{(o)}\rightarrow0,$
i.e. the sequence \emph{(bo)} - converges in   $L_p(\nabla,
 m)$. Now by Theorem 6 \cite[ch.VIII, \S 5]{DSh} for $p=1$ we obtain
 that $s_n(T_\omega)(\Phi(x)(\omega))\stackrel{(o)}\rightarrow \Phi(\tilde{x})(\omega)=v(\omega)$
 in $L_0(\nabla_\omega,m_\omega)$  for
 a.e. $\omega\in\Omega$. Corollary \ref{c3_4} implies that $s_n(T)(x)\stackrel{(o)}\rightarrow\tilde{x}$ in $L_0(\nabla,
 m)$. Using this $(o)$--convergence and order boundedness in $L_p(\nabla,
 m)$ of the sequence $\{s_n(T)({x})\}$ at $p>1$, $x\in L_p(\nabla, m)$ we have that $\tilde{x}\in L_p(\nabla, m)$ and $s_n(T)({x})$ is $(o)$-convergent to $\tilde{x}$
 in
 $L_p(\nabla, m)$.
\end{proof}

\section{Ergodic theorems  in
 Orlicz--Kantorovich lattices $L_M(\nabla,m)$}

Now, we shall present a version of Theorem \ref{th3_5} for
Orlicz--Kantorovich lattices $L_M(\nabla,m)$.

Let $M : R \rightarrow [0,\infty)$ be an $N$-function and let $M^*$
be the complementary $N$-function to $M$ \cite[ch.I, \S\S  1-2]{KR}.

In the same way as in \cite{ZC}, we consider the following subsets in
$L_1(\nabla,m)$:
$$L^0_M(\nabla,m) := L_0(\nabla,m) := \{x\in
L_0(\nabla,m) : M(x)\in L_1(\nabla,m)\},$$
$$L_M(\nabla,m) := \{x \in L_0(\nabla,m) : xy \in
L_1(\nabla,m),  \forall y\in L^0_{M^*}(\nabla,m)\}$$ for which the inclusions
$$L_\infty(\nabla,m)\subset L^0_M(\nabla,m)\subset L_M(\nabla,m)\subset
L_1(\nabla,m)$$ hold.

The set $L_M(\nabla,m)$ is a vector sublattice in $L_1(\nabla,m)$
and with respect to the $L_0(\Omega)$--valued norm
$$ \|x\|_M := \sup\{|\int
xy dm| : y\in L^0_{M^*}(\nabla,m), \int
M^*(y)d\hat{\mu}\leq{\bf1}\}$$ the pair $(L_M(\nabla,m),
\|\cdot\|_M)$ is a Banach--Kantorovich lattice, which is called an
Orlicz--Kantorovich lattice \cite{ZC}.

The following theorem establishes that  Theorem  \ref{th3_2} and  \ref{th3_3} from \cite{ZC}
holds
 for any $L_1$-- $L_\infty$ contractions.
\begin{theorem}\label{th4_1}
If $T\in C_{1,\infty}(\nabla,m)$, then
 \begin{enumerate}
\item[(i)] $T(L_M(\nabla,m))\subset L_M(\nabla,m)$ and $\|T\|_{L_M(\nabla,m)\rightarrow
L_M(\nabla,m)}\leq\mathbf{1}$;

\item[(ii)] if the  $N$--function $M$ meets $\bigtriangleup_2$--condition, then   $s_n(T)(x)$ is $(bo)$-convergent in
$L_M(\nabla,m)$ for every $x\in L_M(\nabla,m)$;

\item[(iii)] if the $N$--function $M$ satisfy $$\sup\limits_{s\geq
1}\frac{\int\limits _1^s M(t^{-1}s)dt}{M(s)}<\infty,$$ then the
sequence $s_n(T)(x)$ is order bounded in Orlicz--Kantorovich lattice
$L_M(\nabla,m)$ for all $x\in L_M(\nabla,m)$ and $s_n(T)(x)$ is
$(o)$-convergent in $L_M(\nabla,m)$.
 \end{enumerate}
\end{theorem}
\begin{proof}
(i)  By \cite[Proposition 2.3]{ZC} we have
that an element $x\in L_1(\nabla,m)$ belongs  to $L_M(\nabla,m)$ if
and only if  $\Phi(x)(\omega)\in L_M(\nabla_\omega,m_\omega)$ for
almost all $\omega\in\Omega$, moreover $ \|x\|_M(\omega)=
\|\Phi(x)(\omega)\|_{L_M(\nabla_\omega,m_\omega)}$ a.e. on
$\Omega$. Since $L_M(\nabla_\omega,m_\omega)$ is an interpolation
 space
 between $L_1(\nabla_\omega,m_\omega)$ and
 $L_\infty(\nabla_\omega,m_\omega)$  \cite[ch. II, \S 4]{kre}, repeating the  proof of Theorem \ref{th3_1}  we obtain that $T_\omega(L_M(\nabla_\omega,m_\omega))\subset
  L_M(\nabla_\omega,m_\omega)$ and $\|T\|_{L_M(\nabla,m)\rightarrow
L_M(\nabla,m)}\leq\mathbf{1}$.

(ii) We need the next properties of the classical Orlicz spaces
$L_M(\nabla_\omega,m_\omega)$ which immediately follow from
 \cite[Proposition 2.1]{suk}.
\begin{claim}\label{cl4_2} Let $N$--function $M$ meets
$\bigtriangleup_2$--condition and let $K$ be a norm bounded set in
$L_M(\nabla_\omega,m_\omega)$. Then $K$ is relative weak compact
if and only if for each $f\in
L^*_M(\nabla_\omega,m_\omega)=L_{M^*}(\nabla_\omega,m_\omega)$ and a
sequence $q_n\in \nabla_\omega$ with $q_n\downarrow 0$ the convergence
$$\sup\left\{\left|\int q_nfhdm_\omega\right|: h\in K\right\}\rightarrow0$$
holds.
\end{claim}

Since $m_\omega(\mathbf{1}_\omega)=1$, Proposition \ref{cl4_2} implies that $\nabla_\omega$ is relatively weak compact in
$L_M(\nabla_\omega,m_\omega)$.

Let $T_\omega$ be an $L_1$--$L_\infty$  contraction in
 $L_1(\nabla_\omega,m_\omega)$ such that $T_\omega(\Phi(x)(\omega)) = \Phi(Tx)(\omega)$ for a.e.
 $\omega\in\Omega$ and  for every
$x\in L_1(\nabla,m)$ (see Theorem  \ref{th3_2}). It is clear that
$\|s_n(T_\omega)\|_{L_M(\nabla_\omega,m_\omega)\rightarrow
L_M(\nabla_\omega,m_\omega)}\leq1$ for all $n\in\mathbb{N}$ and\\
$\|\frac{1}{n}s_n(T_\omega)(h)\|_{L_M(\nabla_\omega,m_\omega)}\rightarrow0$
as $n\rightarrow\infty$ for all $h\in
L_M(\nabla_\omega,m_\omega)$.

Since the $N$--function $M$ meets $\bigtriangleup_2$--condition, the
linear subspace
$$\bigg\{\sum\limits_{i=1}^n \lambda_ie_i : \lambda_i\in \mathbb{R},\ e_i\in\nabla_\omega,\ i=1,\ldots,n,\
 n\in\mathbb{N}\bigg\}$$ is dense in
 $L_M(\nabla_\omega,m_\omega)$, in addition, $\nabla_\omega$ is a
 relatively weak compact set in  $L_M(\nabla_\omega,m_\omega)$.  Hence by
Corollary 3 \cite[ch.VIII, \S 5]{DSh} we have that
 $s_n(T_\omega)(h)$ converges in $L_1(\nabla_\omega,m_\omega)$ for each $h\in
L_1(\nabla_\omega,m_\omega)$. Repeating now the proof of Theorem
\ref{th3_5}(i) we obtain that the sequence $s_n(T)(x)$ $(bo)$--converges in
$L_M(\nabla,m)$ for every $x\in L_M(\nabla,m)$.

(iii)\cite[Theorem 3.3]{ZC} implies that the sequence
$s_n(|T|)(|x|)$ is order bounded in $L_M(\nabla,m)$ for all $x\in
L_M(\nabla,m)$. Repeating  the proof of Theorem \ref{th3_2}(ii) we obtain
that the sequence $s_n(T)(x)$ also is order bounded in
$L_M(\nabla,m)$. Hence from $(o)$--convergence  of sequence
$s_n(T)(x)$ in $L_0(\nabla,m)$ (see Theorem \ref{th3_5}(iii)) it follows its
$(o)$--convergence in $L_M(\nabla,m)$.

\end{proof}

\section{Multiparameter  and weighted ergodic theorems  in
Banach~---Kantorovich lattice
 $L_p(\nabla,m)$}
Let $\mathbb{S}$ be the unit circle in the field $\mathbb{C}$ of
complex numbers and let $\mathbb{Z}$ be the ring of integer numbers.
A function $P_s:\mathbb{Z}\rightarrow\mathbb{C}$ is called a
trigonometric polynomial if
$P_s(k)=\sum\limits_{j=1}^sr_j\lambda_j^k, k\in \mathbb{Z}$, for
some $\{r_j\}_{j=1}^s\subset\mathbb{C}$ and
$\{\lambda_j\}_{j=1}^s\subset\mathbb{S}$. A sequence
$\{\alpha(k)\}$ of complex numbers is called a bounded Besicovich
sequence (BB--sequence) if $\sup\{|\alpha(k)|:
k\in\mathbb{Z}\}<\infty$
 and for every $\varepsilon>0$ there exists a sequence of
trigonometric polynomials $P_s$, such that
$$\lim\limits_{n}\sup\frac{1}{n}\sum\limits_{k=0}^{n-1}|\alpha(k)-P_s(k)|<\varepsilon$$

The following theorem is a vector version of weighted ergodic theorem
for $L_1$--$L_\infty$ contractions in Banach--Kantorovich
lattice
 $L_p(\nabla,m)$ (compare with \cite{GM}).
\begin{theorem}\label{th5_1}
Let  $\{\alpha(k)\}$ be a BB--sequence of real
numbers and $T\in C_{1,\infty}(\nabla,m)$. Then  the averages
$$a_n(T)(x)=\frac{1}{n}\sum\limits_{k=0}^{n-1}\alpha(k)T^k(x)$$
$(o)$-converges in $L_0(\nabla,m)$ for every $x\in L_1(\nabla,m)$.

In addition, if $x\in L_p(\nabla,m)$ and $1<p<\infty$, then the
sequence $a_n(T)(x)$ is order bounded in $L_p(\nabla,m)$ and
$$\|\sup\limits_{n\geq 1} |a_n(T)(x)| \|_{p} \leq
\left(\frac{p}{p-1}\right) \sup\limits_{k}|\alpha(k)\| x \|_{p}.$$
\end{theorem}
\begin{proof}
Let $T_\omega$ be an $L_1$--$L_\infty$ contraction in
$L_1(\nabla_\omega,m_\omega)$ such that $T_\omega(\Phi(x)(\omega))
= \Phi(Tx)(\omega)$  for each a.e. $\omega\in\Omega$ for every
$x\in L_1(\nabla,m)$ (see Theorem \ref{th3_2}).  Since $T_\omega\in
C_{1,\infty}(\nabla_\omega,m_\omega)$, Theorem 1.4
\cite{JO} implies that there exists $v(\omega)\in L_0(\nabla_\omega,m_\omega)$
such that $a_n(T_\omega)(\Phi(x)(\omega))\stackrel{(o)}\rightarrow
v(\omega)$
 in $L_0(\nabla_\omega,m_\omega)$  for
 a.e. $\omega\in\Omega$.  By Corollary \ref{c3_4} there exists $\tilde{x}\in L_0(\nabla,
 m)$ such that $a_n(T)(x)\stackrel{(o)}\rightarrow\tilde{x}$ in $L_0(\nabla,
 m)$.

Now let $1<p <\infty$, $x\in L_p(\nabla, m)$. For every
$\omega\in\Omega$ we consider a positive linear contraction
$S_\omega$ in
 $L_1(\nabla_\omega,m_\omega)$ as in the proof of Theorem \ref{th3_1}.
 Since
 $$|a_n(T)(x)|\leq\frac{1}{n}\sum\limits_{k=0}^{n-1}|\alpha(k)||T^k(x)|\leq \sup\limits_{k}|\alpha(k)|s_n(|T|)(|x|),$$
 we have (see Theorem \ref{th3_5}) that the sequence $a_n(T)(x)$ is order
 bounded in $L_p(\nabla,m)$  and
$$\|\sup\limits_{n\geq 1} |a_n(T)(x)| \|_{p} \leq \sup\limits_{k}|\alpha(k)|\|\sup\limits_{n\geq 1} s_n(|T|)(|x|)\|_{p}\leq
\left(\frac{p}{p-1}\right) \sup\limits_{k}|\alpha(k)\| x \|_{p}. $$
\end{proof}

\begin{remark}\label{r5_2}
 Using  Theorem \ref{th5_1} and repeating proof of
Theorem \ref{th3_2}(i) we obtain that under conditions of Theorem \ref{th5_1} for
every $x\in L_p(\nabla, m)$, $p>1$, there exists $\tilde{x}\in
L_p(\nabla,
 m)$ such  that $a_n(T)(x)\stackrel{(o)}\rightarrow\tilde{x}$ in $L_p(\nabla,
 m)$.
\end{remark}
\begin{remark}\label{r5_2_2}
 If $\{j_k\}_{k=1}^\infty$ is an increasing
sequence of positive integers such that
$\sup\limits_k\frac{j_k}{k}<\infty$ then $$\alpha(k)=\left\{%
\begin{array}{ll}
    0, & \hbox{ if $k\neq j_k$;} \\
    1, & \hbox{ if $k=j_k$} \\
\end{array}%
\right.$$  is a BB--sequence. Therefore, Theorem \ref{th5_1} implies
that for any $T\in C_{1,\infty}(\nabla,m)$, ${x}\in L_p(\nabla,
 m)$,  the average  $a_n(T)(x)=\frac{1}{n}\sum\limits_{k=0}^{n-1}T^{j_k}(x)$
 $(o)$-converges
 in $L_0(\nabla,m)$. In addition, if $p>1$, $x\in L_p(\nabla, m)$, then average  $\frac{1}{n}\sum\limits_{k=0}^{n-1}T^{j_k}(x)$
 $(o)$-converges
 in $L_p(\nabla,m)$.
\end{remark}

By Theorems \ref{th2_2}(i), \ref{th5_1} and Remark \ref{r5_2} we obtain the following
\begin{corollary}\label{c5_4} Let  $\{\alpha(k)\}$ be a BB--sequence of
real numbers and $T\in C_{1,\infty}(\nabla,m)$. Then  the averages
$a_n(T)(x)$ $(bo)$-converge in $L_p(\nabla,m)$ for every $x\in
L_p(\nabla,m)$.
\end{corollary}

Now we present a version of multiparameter  ergodic theorem for $L_1$--$L_\infty$ contractions in $L_p(\nabla,m)$.

Firstly, let us recall the  definition of $d$--dimensional analog
BB-sequence \cite{JO}.
Let $d\in \mathbb{N}$,
$\mathbf{n}=(n_1,n_2,\dots,n_d)\in\mathbb{Z}^d$, ${\bf 1}\!\!{\rm
I}=(1,1,\dots,1),$ ${\bf 0}=(0,0,\dots,0).$ The product of all
nonzero components of $\mathbf{n}=(n_1,n_2,\dots,n_d)$ we denote by
$|\mathbf{n}|$. We say that $\{\alpha(\mathbf{k})\}\in\mathbb{C}$
is BB-sequence if $\sup\{|\alpha(\mathbf{k})|:
\mathbf{k}\in\mathbb{Z}^d\}<\infty$ and  for every $\varepsilon>0$
there exists a sequence of trigonometric polynomials in $d$
variables  $\mathbf{P}_\varepsilon$ such that
$$\limsup\limits_{\mathbf{n}\rightarrow\infty}\frac{1}{\mathbf{|n|}}
\sum\limits_{{\mathbf{k}}=\mathbf{0}}^{{\mathbf{n}}-{\bf
1}\!\!{\rm
I}}|\alpha(\mathbf{k})-\mathbf{P}_\varepsilon(\mathbf{k})|<\varepsilon.$$

Let $\mathbb{T}=\{T_1,T_2,\dots,T_d\}$ denotes  a family of  $d$
linear operators mapping  $L_1(\nabla,m)$ into itself such that
$T_i\in C_{1,\infty}(\nabla,m)$, $i=1,\cdots,n$.  Let
$\{\alpha(\mathbf{k})\}$ be a BB-sequence of real numbers. For any
$x\in L_1(\nabla,m)$, $\mathbf{n}\in(\mathbb{N}\cup \{0\})^d$,
$\mathbf{n}\neq\mathbf{0}$ we set
\begin{eqnarray*}
A_{\mathbf{n}}(\mathbb{T})(x)=\frac{1}{n_1\cdot n_2\cdots
n_d}\sum\limits_{k_1=0}^{n_1-1}\cdots\sum\limits_{k_d=0}^{n_d-1}\alpha(k_1,k_2,\cdots,k_d)T_1^{k_1}\cdots
T_d^{k_d}(x)\\=\frac{1}{\mathbf{|n|}}
\sum\limits_{{\mathbf{k}}
=\mathbf{0}}^{{\mathbf{n}}-{\bf
1}\!\!{\rm
I}}\alpha(\mathbf{k})\mathbb{T}^{\mathbf{k}}(x),
\end{eqnarray*}
where
$\mathbb{T}^{\mathbf{k}}(x)=T_1^{k_1}\cdots T_d^{k_d}(x).$

\begin{theorem}\label{th5_5}
Let $T_i\in C_{1,\infty}(\nabla,m)$,
 $i=1,\cdots,d$, $1<p<\infty$ and let $\{\alpha(\mathbf{k})\}$ be a $d$--dimensional  BB-sequence of real numbers.
Then averages $A_{\mathbf{n}}(\mathbb{T})(x)$ $(o)$- converge  in
$L_p(\nabla,m)$ for every $x\in L_p(\nabla,m)$, in addition
$$\|\sup\{ |A_{\mathbf{k}}(\mathbb{T})(x)|: \mathbf{k}\in(\mathbb{N}\cup \{0\})^d\}\|_{p} \leq
\left(\frac{p}{p-1}\right)^d
\sup\limits_{\mathbf{k}}|\alpha(\mathbf{k})| \| x \|_{p}.\eqno(1)
$$
\end{theorem}
\begin{proof}
Let $(T_i)_\omega$ be an $L_1$--$L_\infty$ contraction in
$L_1(\nabla_\omega,m_\omega)$ such that $(T_i)_\omega(\Phi(x)(\omega))
= \Phi(T_i(x))(\omega)$  for each a.e. $\omega\in\Omega$ and for every
$x\in L_1(\nabla,m)$, $i=1,\cdots,d$. Since $(T_i)_\omega \in
C_{1,\infty}(\nabla_\omega,m_\omega)$ by Theorem 1.2
\cite{JO} there exists $v(\omega)\in L_0(\nabla_\omega,m_\omega)$
such that  $A_{\mathbf{n}}(\mathbb{T}_\omega)(\Phi(x)(\omega))\stackrel{(o)}\rightarrow
v(\omega)$
 in $L_0(\nabla_\omega,m_\omega)$, where  $\mathbb{T}_\omega=((T_1)_\omega,\cdots,(T_d)_\omega)$.
Since  $A_{\mathbf{n}}(\mathbb{T}_\omega)(\Phi(x)(\omega)) =
\Phi(A_{\mathbf{n}}(\mathbb{T})(x))(\omega)$  for  a.e.
$\omega\in\Omega$ Corollary \ref{c3_4} implies that there exists
$\tilde{x}\in L_0(\nabla,
 m)$ such that $A_{\mathbf{n}}(\mathbb{T})(x)\stackrel{(o)}\rightarrow\tilde{x}$ in $L_0(\nabla,
 m)$.

 We consider the family $|\mathbb{T}|=\{|T_1|,|T_2|,\dots,|T_d|\}\subset C_{1,\infty}(\nabla,m).$

Let $(S_i)_\omega$ be a positive contraction in
$L_p(\nabla_\omega,m_\omega)$ such that
$(S_i)_\omega(\Phi(x)(\omega)) = \Phi(|T_i|x)(\omega)$ for  a.e.
$\omega\in\Omega$ (see the proof of Theorem \ref{th3_1}). By
Theorem 1.2 \cite[ch.YI, \S 6.1]{Kr} we have that sequence
$$B_{\mathbf{n}}(\mathbb{S}_\omega)(|\Phi(x)(\omega)|)=\frac{1}{\mathbf{|n|}}
\sum\limits_{{\mathbf{k}}=\mathbf{0}}^{{\mathbf{n}}-{\bf
1}\!\!{\rm I}}(S_1)_\omega^{k_1}\cdots
(S_d)_\omega^{k_d}(|\Phi(x)(\omega)|)$$
 is order bounded  in
$L_p(\nabla_\omega,m_\omega)$ and
$$\|\sup\{
B_{\mathbf{n}}(\mathbb{S}_\omega)(|\Phi(x)(\omega)|):
\mathbf{k}\in(\mathbb{N}\cup
\{0\})^d\}\|_{L_p(\nabla_\omega,m_\omega)} \\\leq
\left(\frac{p}{p-1}\right)^d\|\Phi(x)(\omega)\|_{L_p(\nabla_\omega,m_\omega})
$$ for  a.e. $\omega\in\Omega$, where
$\mathbb{S}_\omega=\{(S_1)_\omega,\cdots,(S_d)_\omega\}.$ Since
$|(T_i)_\omega g|\leq(S_i)_\omega|g|$ for all $g\in
L_1(\nabla_\omega,m_\omega)$(see the proof of Theorem \ref{th3_2}) we obtain
that
$$|A_{\mathbf{n}}(\mathbb{T}_\omega)(\Phi(x)(\omega))|\leq
\frac{1}{\mathbf{|n|}}\sum\limits_{k_1=0}^{n_1-1}\cdots\sum\limits_{k_d=0}^{n_d-1}|\alpha(\mathbf{k})||(T_1)_\omega^{k_1}|\cdots
|(T_d)_\omega^{k_d}|(|\Phi(x)(\omega)|)\leq$$
$$\leq \sup\limits_{\mathbf{k}}|\alpha(\mathbf{k})|\sum\limits_{k_d=0}^{n_d-1}(S_1)_\omega^{k_1}
\cdot\cdots\cdot(S_d)_\omega^{k_d}(|\Phi(x)(\omega)|)=
\sup\limits_{\mathbf{k}}|\alpha(\mathbf{k})|B_{\mathbf{n}}(\mathbb{S}_\omega)(|\Phi(x)(\omega)|).$$

Using now Theorem \ref{th3_3}, we have that the sequence
$A_{\mathbf{n}}(\mathbb{T})(x)$ is order bounded in $L_p(\nabla,
 m)$ and  the inequality $(1)$ holds.  Furthermore, the $(o)$--convergence  $A_{\mathbf{n}}(\mathbb{T})(x)\stackrel{(o)}\rightarrow\tilde{x}$ in $L_0(\nabla,
 m)$ implies  that $A_{\mathbf{n}}(\mathbb{T})(x)$
 $(o)$--converges to $\tilde{x}$ in $L_p(\nabla,
 m)$.

\end{proof}

\end{document}